\documentclass[12pt]{article}
\usepackage[utf8]{inputenc}
\usepackage[T1]{fontenc}

\usepackage{moresize}

\usepackage[english]{babel}

\usepackage{lmodern}

\usepackage{enumerate}

\usepackage{amsmath, amsfonts, amssymb, amsthm}
\usepackage{ dsfont }
\usepackage{caption}
\usepackage{subcaption}
\usepackage{mleftright}

\usepackage[linesnumbered,ruled,lined]{algorithm2e}

\usepackage{xcolor} 
\definecolor{Prune}{RGB}{99,0,60} 

\usepackage{color}
\definecolor{corange}{rgb}{0.898, 0.621, 0.0}
\definecolor{cskyblue}{rgb}{0.336, 0.703, 0.910}
\definecolor{cbluishgreen}{rgb}{0, 0.617, 0.449}
\definecolor{cyellow}{rgb}{0.937, 0.890, 0.258}
\definecolor{cblue}{rgb}{0, 0.445, 0.695}
\definecolor{cred}{rgb}{1, 0, 0}
\definecolor{cpurple}{rgb}{0.797, 0.473, 0.652}

\usepackage{mdframed}
\usepackage{multirow} 
\usepackage{multicol} 
\usepackage{scrextend} 

\usepackage{tikz}
\usepackage{graphicx,calc}
\usepackage[absolute]{textpos} 
\usepackage{colortbl}
\usepackage{array}
\usepackage[margin=2cm]{geometry}
\usepackage{titlesec}
\usepackage[colorlinks=true,
linkcolor=cbluishgreen,
urlcolor=cpurple,
citecolor=cbluishgreen,
backref=page,
linktocpage]{hyperref}

\theoremstyle{definition}
\newtheorem{theorem}{Theorem}[section] 
\newtheorem{proposition}[theorem]{Proposition} 
\newtheorem{lemma}[theorem]{Lemma} 
\newtheorem{corollary}[theorem]{Corollary}

\newtheorem{definition}[theorem]{Definition}

\newtheorem{example}[theorem]{Example}
\newtheorem{remark}[theorem]{Remark}

\def\RR{\mathbb{R}}

\def\cBT{\mathcal{BT}}
\def\cC{\mathcal{C}}

\def\cPT{\mathcal{PT}}

\def\fS{\mathfrak{S}}

\DeclareMathOperator{\PAsoc}{Assoc}
\DeclareMathOperator{\PCube}{Cube}
\DeclareMathOperator{\PPT}{PT}

\newcommand{\gbinom}[2]{{\mleft(\genfrac..{0pt}{}{#1}{#2}\mright)}}
\newcommand\bbinom[2]%
{\mathchoice{\left(\kern-0.48em{\binom{#1}{#2}}\kern-0.48em\right)}
            {\bigl(\kern-0.3em{\binom{#1}{#2}}\kern-0.3em\bigr)}
            {\bigl(\kern-0.3em{\binom{#1}{#2}}\kern-0.3em\bigr)}
            {\bigl(\kern-0.3em{\binom{#1}{#2}}\kern-0.3em\bigr)}}

\newlength\myheight{}
\newlength\mydepth{}
\settototalheight\myheight{Xygp}
\DeclareRobustCommand*\nonee{\raisebox{-\mydepth}{\includegraphics[height=\myheight]{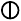}}}
\DeclareRobustCommand*\upp{\raisebox{-\mydepth}{\includegraphics[height=\myheight]{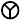}}}
\DeclareRobustCommand*\downn{\raisebox{-\mydepth}{\includegraphics[height=\myheight]{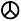}}}
\DeclareRobustCommand*\uppdownn{\raisebox{-\mydepth}{\includegraphics[height=\myheight]{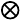}}}

\newcommand{\defn}[1]{\emph{\color{cblue} #1}} 

\title{Inversion and Cubic Vectors for Permutrees}

\author{Daniel Tamayo Jiménez}

\newcommand{\Addresses}{{
  \bigskip
  \footnotesize

  (D. Tamayo Jim\'enez), \textsc{Universit\'e Paris-Saclay, GALaC, Gif-sur-Yvette, France}\par\nopagebreak
  \textit{Email adress:}, \texttt{daniel.tamayo-jimenez@lri.fr}\par\nopagebreak
  \textit{URL:} \texttt{\url{https://sites.google.com/view/danieltamayo22/}}

}}

\date{\today}

\begin{document}

\maketitle


\begin{abstract}
	We introduce two generalizations of bracket vectors from binary trees to permutrees. These new vectors help describe algebraic and geometric properties of the rotation lattice of permutrees defined by Pilaud and Pons. The first generalization serves the role of an inversion vector for permutrees allowing us to define an explicit meet operation and provide a new constructive proof of the lattice property for permutree rotation lattices. The second generalization, which we call cubic vectors, allows for the construction of a cubic realization of these lattices which is proven to form a cubical embedding of the corresponding permutreehedra. These results specialize to those known about permutahedra and associahedra.
\end{abstract}

\section{Introduction}

Permutrees are combinatorial objects that generalize and amalgam binary trees, permutations, binary words, and Cambrian trees. In the same way that binary trees can be seen as labelled trees with one parent and two children, permutrees are labelled oriented trees whose vertices can have one or two parents and one or two children. They were introduced by Pilaud and Pons in~\cite{PP18} where they addressed the combinatorial, geometric, and algebraic features of these trees. They constructed lattices using congruences of the weak order based on~\cite{R06}, polytopes called permutreehedra that encapsulate permutahedra and associahedra, and Hopf algebras that contain the ones studied in~\cite{CP17},~\cite{GKLL95},~\cite{LR98}, and~\cite{MR95}.

In this paper, inspired by~\cite{C22} and~\cite{HT72}, we give two generalizations of the bracket vector for binary trees to obtain algebraic and geometric results about permutree rotation lattices. Bracket vectors were first introduced by S. Huang and D. Tamari in~\cite{HT72} to construct a meet operation in the rotation order on binary trees. This allowed them to obtain a simple proof of the lattice property of this poset. Due to their simplicity, bracket vectors have been used through the years via other generalizations such as in~\cite{CPS19},~\cite{C22},~\cite{FMN21}, and~\cite{P86}.

Our first generalization of bracket vectors, which we call inversion vectors, gives us a characterization of permutrees through inversion sets (Lemma~\ref{lem:permutree_inversion_sets}). With this we construct a meet operation between permutrees (Theorem~\ref{thm:permutree_meet}) and obtain a new constructive proof of the lattice property of the poset of rotations on permutrees (Corollary~\ref{cor:permutrees_are_lattices}). This unifies and extends the known results on permutations and binary trees along with binary words and Cambrian trees, to any interpolation between these objects.

The second generalization called cubic vectors, expands the inversion vector and then places it in space such that the resulting structure is a cubical realization of the permutree rotation lattice. That is, an embedding of the respective permutreehedra onto a stretched cube of its same dimension (Theorem~\ref{thm:cubic_property_embedding}). This specifies to the cubic realizations of~\cite{C22} on the associahedron when restricted to binary trees elements instead of general Tamari intervals and those of~\cite{BF71} and~\cite{RR02} when dealing with the permutahedron.

\subsection{Binary Trees and Bracket Vectors}\label{ssec:binary_trees}

Throughout this work we denote the set $\{1,\ldots,n\}$ as $[n]$.

\defn{Binary trees} are rooted planar trees whose vertices have one parent and two children. They are well studied combinatorial objects counted by the \defn{Catalan numbers}. One can label the vertices of a binary tree from $1$ to $n$, by taking an anti-clockwise walk of the tree starting at its root, and labelling a vertex when it is visited for the second time. This is known as the \defn{inorder labelling}. With this labelling we denote~$L_i$ (resp.~$R_i$) the left (resp.\ right) subtree with of the vertex labelled by~$i$. As in~\cite{P86}, one can induce a partial order on inordered binary trees through the rotation operation on the edges described in Figure~\ref{fig:binary_tree_rotation} where~$1\leq i<j\leq n$. Figure~\ref{fig:tamari_3_bracket_sets} presents the resulting poset when $n=3$ with the inorder labelling of each binary tree.

\begin{figure}[h!]
	\centering
	\includegraphics[scale=1.5]{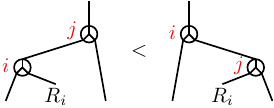}
	\caption{Rotation on binary trees.}\label{fig:binary_tree_rotation}
\end{figure}

Huang and Tamari gave a constructive proof in~\cite{HT72} that showed that these posets always have a meet and therefore are in fact lattices. Due to this, these posets are commonly referred as \defn{Tamari Lattices} and here we denote them by $\mathcal{BT}_n$. The meet construction they gave relies on a bijection between binary trees on $n$ vertices and bracketing functions $f:[n]\rightarrow [n]$. These bracketing functions consist on a vector called the \defn{bracket vector} that records in its $i$-th entry all elements between $i$ and $f(i)$ including $f(i)$. We forgo this definition and instead use the following equivalent formulation.

\begin{definition}\label{def:bracket_vector}
	Let~$T$ be an inordered binary tree with vertex set~$[n]$. Its \defn{bracket set} is~$B(T):=\{(i,j)\,:\,j\in R_i\}$ and its \defn{bracket components} are~${B(T)}_i=\{j\in [n] \,:\, (i,j)\in B(T)\}$.
	To a bracket set we associate a \defn{bracket vector}~$\vec{b}(T)=(b_1,\ldots,b_{n-1})$ where~$b_i=|{B(T)}_i|$.
\end{definition}

Notice that we do not consider~${B(T)}_n$ as~$R_n=\emptyset$. Figure~\ref{fig:tamari_3_bracket_sets} presents all bracket sets for~$\cBT_3$. It is possible to characterize which vectors are bracket vectors of binary trees with our terminology.

\begin{figure}[h!]
	\centering
	\includegraphics[scale=0.8]{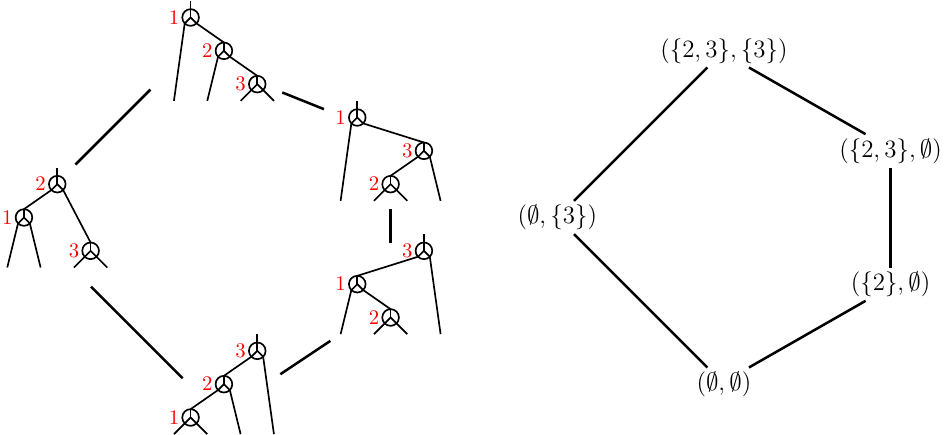}
	\caption{ The Tamari lattice for~$n=3$ and its corresponding bracket sets represented by their components.}\label{fig:tamari_3_bracket_sets}
\end{figure}

\begin{proposition}[{\cite{HT72}}]\label{prop:bracket_vector_characterization}
	The bracket set map is a bijection from binary trees to the sets~$B\in 2^{[n]}$ such that their components satisfy
	\begin{enumerate}
		\itemsep0em
		\item $B_i=\emptyset$ or~$B_i=\{i+1,i+2,\ldots,i+l\}$ for some~$l>0$,

		\item if~$j\in B_i$, then~$B_j\subseteq B_i$.
	\end{enumerate}
\end{proposition}

The bracket set has similarities to inversions of permutations. As such, it is the key ingredient in formulating the meet operation as follows.

\begin{proposition}[{\cite{HT72}}]\label{prop:tamari_lattice_meet}
	Given two binary trees~$T,T'$ with~$n$ vertices, there exists the binary tree~$T\wedge T'$ under the rotation order. Moreover, it satisfies for all~$i\in[n-1]$ \begin{equation}\label{eq:binary_meet}{B(T\wedge T')}_i={B(T)}_i\cap {B(T')}_i.\end{equation}
\end{proposition}

Figure~\ref{fig:tamari_meet} illustrates the meet operation between two binary trees in~$\cBT_5$. Calculating their bracket sets via Equation~\ref{eq:binary_meet} yields~$(\{2,3,4,5\},\emptyset,\emptyset,\{5\})\cap(\emptyset,\emptyset,\{4,5\},\{5\})=(\emptyset,\emptyset,\emptyset,\{5\})$.

\begin{figure}[h!]
	\centering
	\includegraphics[scale=1]{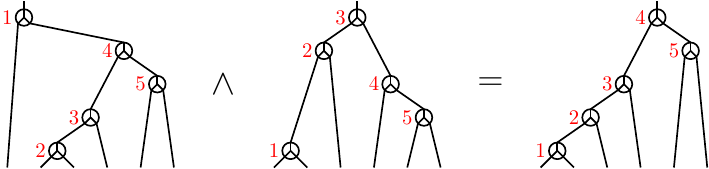}
	\caption{ The meet of two binary trees.}\label{fig:tamari_meet}
\end{figure}

\begin{corollary}[{\cite{HT72}}]\label{thm:tamari_lattice_proof}
	The poset of binary trees~$(\cBT_n,\leq)$ is a lattice.
\end{corollary}

The proof of Corollary~\ref{thm:tamari_lattice_proof} follows from the meet construction and the fact that the poset is bounded. See~\cite[Prop.3.3.1]{S11} or~\cite[Lem.9-2.1]{R16} for proofs of this fact.

\enskip{}

On top of the rotation poset being a lattice, the Tamari lattice also appears as the~$1$-skeleton of the \defn{associahedron}~$\PAsoc_n$  constructed equivalently as \begin{itemize}
	\item \cite{L04} the convex hull of the coordinates~$(l_1r_1,\ldots,l_nr_n)$ for~$T\in{\cBT}_n$ where~$l_i$ (resp.~$r_i$) is the number of leaves of~$L_i$ (resp.~$R_i$),
	\item \cite{SS93} the intersection of the following hyperplane and half-spaces \begin{equation*}
		      \left\{\mathbf{x}\in\RR^n \,:\,\sum_{i\in [n]}x_i=\gbinom{n+1}{2}\right\} \cap \bigcap_{1\leq i \leq j\leq n} \left\{\mathbf{x}\in\RR^n \,:\, \sum_{i\leq \ell\leq j}x_\ell\geq\gbinom{j-i+2}{2}\right\}.
	      \end{equation*}
\end{itemize}

Figure~\ref{fig:asoc4} contains an example of the associahedron in 3 dimensions.

\begin{figure}[h!]
	\centering
	\includegraphics[scale=0.6]{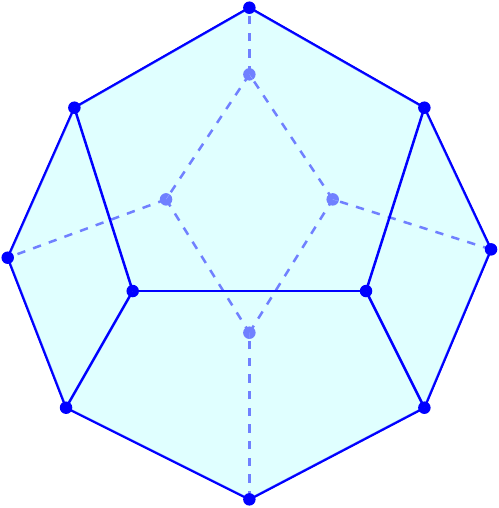}
	\caption{The associahedron~$\PAsoc_4$.}
	\label{fig:asoc4}
\end{figure}

The bracket vector also allows for a particular geometric realization of the Tamari lattice by defining a cubical embedding of the associahedron as follows.

\begin{proposition}[{\cite{P86},\cite{BW96}}]\label{:cubic_associahedron}
	The associahedron~$\PAsoc_n$ is embeddable in the stretched cube~$Q_{n-1}=[0,n-1]\times\cdots\times[0,1]$ via the function that sends a binary tree to its bracket vector.
\end{proposition}

Although this way of representing the associahedron has been known since the 80s (see~\cite{P86} and~\cite{BW96}), the explicit illustration of this embedding as an actual cube seems to date from the video~\cite{K93}. We invite (in genuine interest) the archaeological reader to find an older illustration of this embedding into~$Q_{n-1}$. This cubic phenomenon has appeared in recent works concerning Tamari intervals~\cite{C22} and parabolic Tamari lattices in Coxeter groups of type~$B$~\cite{FMN21}.

In Figure~\ref{fig:associahedron_cubic} we show the cubical embedding of~$\PAsoc_4$ following~\cite{K93}.

\begin{figure}[h!]
	\centering
	\includegraphics[scale=0.6]{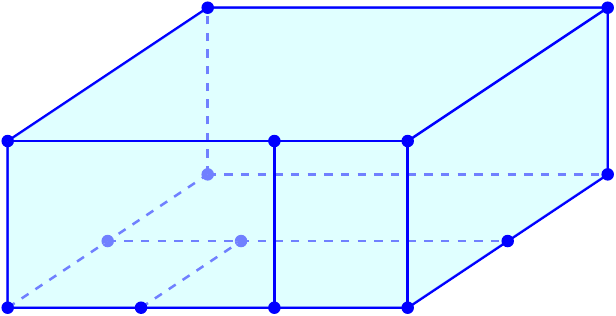}
	\caption{ The cubical embedding of~$\PAsoc_4$. Figure based on the video~\cite{K93}.}\label{fig:associahedron_cubic}
\end{figure}

\subsection{Permutrees}\label{ssec:permutrees}

Binary trees are part of a more general family of  combinatorial objects called permutrees. Defined by Pons and Pilaud in~\cite{PP18}, they generalize permutations and binary trees in such a way that they also capture binary sequences and \defn{Cambrian trees} (see~\cite{LP13} and~\cite{CP17}) that were motivated by the Cambrian lattices of~\cite{R06}.

A \defn{permutree} is a directed planar unrooted tree~$T$ with vertex set~$\{v_1,\ldots,v_n\}$ such that for each vertex~$v_i$:
\begin{enumerate}
	\itemsep0em
	\item $v_i$ has exactly one or two parents (outward neighbors) and one or two children (inward neighbors). We denote respectively~$LA_i$,~$RA_i$, (resp.~$LD_i$,~$RD_i$) the left and right ancestor (resp.\ descendant) subtree of~$v_i$. In the case that a vertex has only one ancestor (resp.\ descendant) subtree we denote it~$A_i$ (resp.~$D_i$),

	\item if~$v_i$ has two parents (resp.\ children), then all vertices~$v_j\in LA_i$ (resp.~$v_j\in LD_i$) satisfy~$j<i$ and all vertices~$v_k\in RA_i$ (resp.~$v_k\in RD_i$) satisfy~$i<k$.
\end{enumerate}
If~$v_j$ is a descendant of~$v_i$ we say that \defn{$j\to i$}. Given a permutree~$T$, its \defn{partial order} on~$[n]$ is given by~$j<i$ if and only if~$j\to i$.

The \defn{decoration} of a permutree~$T$ is the vector~$\delta(T)\in{\{\nonee,\downn,\upp,\uppdownn\}}^n$ with entries defined as
	\begin{equation*}
		{\delta(T)}_i=\left\{\begin{array}{ccl}
			\nonee    & \text{ if } & v_i \text{ has one parent and one child,}     \\
			\downn    & \text{ if } & v_i \text{ has one parent and two children,}  \\
			\upp      & \text{ if } & v_i \text{ has two parents and one child,}    \\
			\uppdownn & \text{ if } & v_i \text{ has two parents and two children.}
		\end{array}\right.
	\end{equation*} Letting~$\delta:=\delta(T)$ we say that~$T$ is a~$\delta$-permutree and denote by \defn{$\cPT_n(\delta)$} the collection of all~$\delta$-permutrees on~$n$ vertices.

\begin{example}\label{ex:permutree_examples}
	Permutrees~$\cPT_n(\delta)$ correspond to: \begin{itemize}
		\itemsep0em
		\item permutations when~$\delta=\nonee^n$,
		\item binary trees when~$\delta=\downn^n$,
		\item Cambrian trees when~$\delta\in{\{\downn,\upp\}}^n$,
		\item binary sequences of length~$n-1$ when~$\delta=\uppdownn^n$ via the correspondence that the coordinates of the binary sequence are~$s_i=0$ (resp.~$s_i=1$) if the vertex~$v_i$ is a child (resp.\ parent) of~$v_{i+1}$.
	\end{itemize}

	Figure~\ref{fig:permutree-examples} contains several examples of permutrees with distinct decorations.
\end{example}

\begin{figure}[h!]
	\centering
	\includegraphics[scale=1.2]{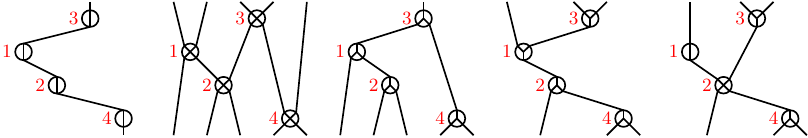}
	\caption{ 5 examples of~$\delta$-permutrees on 4 vertices. These permutrees respectively correspond to the permutation~$4213$, the binary sequence~$101$, a rooted binary tree, a Cambrian tree and a generic permutree.}\label{fig:permutree-examples}
\end{figure}

\begin{remark}\label{rem:permutree_leftmost_rightmost_labels}
	Notice that the decorations~$\delta_1$ and~$\delta_n$ do not actually affect the structure of the~$\delta$-permutree since the subtrees~$LA_1$,~$LD_1$,~$RA_n$, and~$RD_n$ are always empty. We never make use of these subtrees, so we always take~$\delta_1=\delta_n=\nonee$ for simplicity.
\end{remark}

When drawing~$\delta$-permutrees, their definition allows us to present them in a non-ambiguous way. All edges are assumed to be directed upwards and thus, they are presented unoriented, the vertices~$v_i$ appear from left to right in ascending order. This follows from the \defn{insertion algorithm} of~\cite{PP18}.

Like for binary trees, for any fixed decoration~$\delta$ one can define edge rotations on~$\delta$-permutrees.

\begin{definition}\label{def:permutree_rotations}
	Let~$T\in\cPT_n(\delta)$ be a~$\delta$-permutree with an edge~$i\to j$ where~$1\leq i<j\leq n$. An \defn{ij-edge rotation} is the operation of replacing the (right) subtree of~$v_i$ by~the (left) subtree of~$v_j$ and the (left) subtree by the tree with root~$v_i$, maintaining rest of~$T$ intact. Figure~\ref{fig:binary_tree_rotation} shows all possible~$ij$-edge rotations.

	The \defn{edge cut} in~$T$ defined by~$i\to j$ is the ordered partition~$(I\,\|\, [n]\setminus{I})$ of the vertex set of~$T$ where~$I$ are the vertices whose undirected paths to~$v_i$ do not visit~$v_j$.
\end{definition}

\begin{example}\label{ex:permutree_edge_cut}
	Consider the~$\nonee\uppdownn\upp\downn$-permutree given at the far right Figure~\ref{fig:permutree-examples}. The respective edge cuts of the directed edges~$2\to 1$,~$2\to 3$, and~$4\to 2$, are~$(\{2,3,4\}\,\|\, \{1\})$,~$(\{1,2,4\}\,\|\, \{3\})$, and~$(\{4\}\,\|\, \{1,2,3\})$.
\end{example}

\begin{proposition}\label{rotation_edge_cuts}
	The~$ij$-rotation of a~$\delta$-permutree~$T$ is a~$\delta$-permutree~$T'$ whose edge cuts are precisely those of~$T$ except the edge cut defined by~$i\to j$.
\end{proposition}

The resulting poset is called the rotation poset of~$\delta$-permutrees and its covering relations are characterized by edge rotations. Figure~\ref{fig:permutree_rotations} shows rotations between all possible adjacent vertices and Figure~\ref{fig:permutree-inversion-vector-IXYI} presents an example of such a rotation poset where~$\delta=\nonee\uppdownn\upp\nonee$.

\begin{figure}[h]
	\centering
	\includegraphics[scale=0.905]{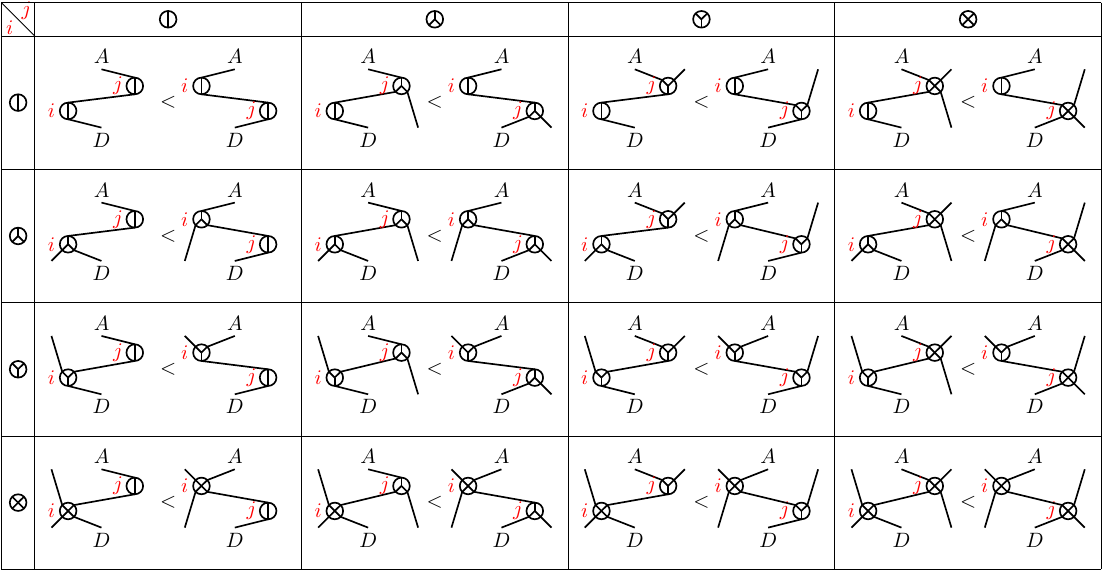}
	\caption{ All possible~$ij$-rotations of~$\delta$-permutrees. Figure based from~\cite{PP18}.}\label{fig:permutree_rotations}
\end{figure}

\begin{remark}\label{rem:permutree_rotations_bounded}
	Notice that~$\delta$-permutree posets are always bounded. The minimal element~$\hat{0}_\delta$ (resp.\ maximal element~$\hat{1}_\delta$) is the~$\delta$-permutree such that~$i\to i+1$ (resp.~$i+1\to i$) for all~$i\in[n-1]$.
\end{remark}

As for binary trees, the rotation poset of permutrees is a lattice.

\begin{proposition}[{\cite[Prop.2.32]{PP18}}]\label{prop:permutree_lattice_property}
	The poset of~$\delta$-permutrees~$(\cPT_n(\delta),\leq)$ is a lattice.

	Moreover, the~$\delta$-permutree lattice is isomorphic to \begin{itemize}
		\itemsep0em
		\item the weak order of~$\fS_n$ if~$\delta=\nonee^n$,
		\item the Tamari lattice if~$\delta=\downn^n$,
		\item the (Type~$A$) Cambrian lattices if~$\delta\in\{\downn,\upp\}^n$,
		\item the boolean lattice if~$\delta=\uppdownn^n$.
	\end{itemize}
\end{proposition}

The proof of~\cite{PP18} of the lattice property uses the theory of lattice quotients. In Section~\ref{sec:inversion_vectors} we give constructive proof of this fact using similar ideas as Proposition~\ref{prop:tamari_lattice_meet}.

Like the Tamari lattice, the~$\delta$-permutree rotation lattice appears as the~$1$-skeleton of the~\defn{$\delta$-permutreehedron}~$\PPT(\delta)$ constructed equivalently as \begin{itemize}
  \itemsep0em
  \item the convex hull of points of the form \begin{equation*}
          \mathbf{a}(T)_i =
          \begin{cases}
            1+d                           & \text{ if } \delta_i=\nonee,   \\
            1+d+|LD_i||RD_i|              & \text{ if } \delta_i=\downn,   \\
            1+d-|LA_i||RA_i|              & \text{ if } \delta_i=\upp,     \\
            1+d+|LD_i||RD_i|-|LA_i||RA_i| & \text{ if }\delta_i=\uppdownn, \\
          \end{cases}
        \end{equation*} where~$d$ is the number of descendants of~$v_i$, and~$T$ is a~$\delta$-permutree,
  \item the intersection of the following hyperplane and half-spaces \begin{equation*}
          \left\{\mathbf{x}\in\RR^n \,:\,\sum_{i\in [n]}x_i=\gbinom{n+1}{2}\right\} \cap \bigcap_{I\in\mathcal{I}} \left\{\mathbf{x}\in\RR^n \,:\, \sum_{i\in I}x_i\geq\gbinom{|I|+1}{2}\right\},
        \end{equation*} where~$\mathcal{I}=\{I\subsetneq[n]\,:\,\exists\text{ a~$\delta$-permutree with edge cut }(I\,\|\,[n]\setminus I)\}$.
\end{itemize} 

See~\cite{PP18} for more details of this geometric construction. Figure~\ref{fig:PPT4} contains examples of~$\delta$-permutreehedra for some decorations.

\begin{figure}[h!]
	\centering
	\includegraphics[scale=1.2]{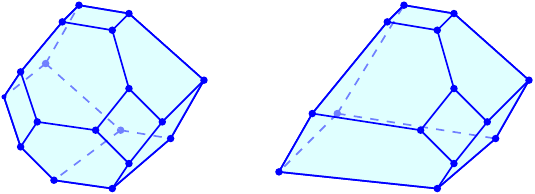}
	\caption{ The permutreehedra~$\PPT(\nonee\upp\nonee\nonee)$ (left) and~$\PPT(\nonee\upp\downn\nonee)$ (right).}
	\label{fig:PPT4}
\end{figure}

\section{Inversion Vectors}\label{sec:inversion_vectors}

We begin defining inversion vectors for~$\delta$-permutrees with the aim of proving the lattice property of~$\delta$-permutrees rotation posets (Proposition~\ref{prop:permutree_lattice_property}) in a constructive manner.

Recall that~$j\to i$ if~$v_j$ is a descendant of~$v_i$.

\begin{definition}\label{def:inversion_vector_permutrees}
	Consider~$T\in\cPT_n(\delta)$ to be a~$\delta$-permutree. Its \defn{inversion set} and \defn{inversion components} are \begin{equation*}
		\begin{split}
			B(T)&:=\{(i,j)\,:\, i<j  \text{ and } j\to i \},\\
			{B(T)}_i&:=\{j\in [n] \,:\, (i,j)\in B(T)\}.
		\end{split}
	\end{equation*} That is, all $j>i$ such that~$v_j$ is a descendant of~$v_i$. An inversion set has an associated \defn{inversion vector}~$\vec{b}(T)=(b_1,\ldots,b_{n-1})$ such that~$b_i=|{B(T)}_i|$.
\end{definition}

Since~$RD_{n}=\emptyset$, its component does not alter the combinatorics of inversion sets and thus, we do not consider it. Figure~\ref{fig:permutree-inversion-vector-IXYI} contains the inversion vectors for all~$\nonee\uppdownn\upp\nonee$-permutrees.

\begin{figure}[h!]
	\centering
	\includegraphics[scale=0.6]{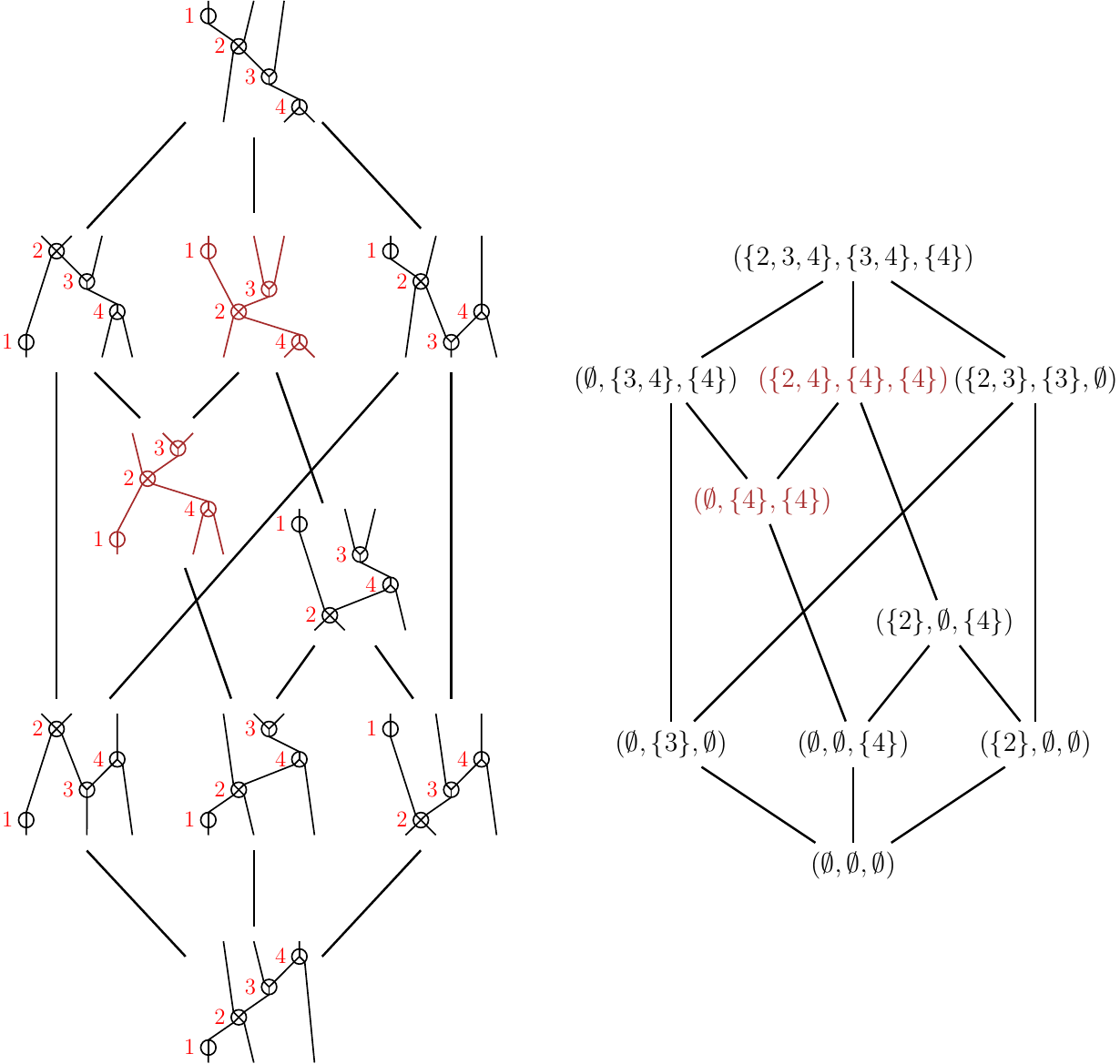}
	\caption{ The rotation lattice of~$\nonee\uppdownn\upp\downn$-permutrees with their respective inversion sets represented via their components. The permutrees in brown are those that are not extremal.}\label{fig:permutree-inversion-vector-IXYI}
\end{figure}

\begin{example}
	Let~$\hat{1},\,\hat{0},\, T_l,$ and~$T_r$ respectively be the top, bottom, middle left, and middle right elements of the lattice of~$\nonee\uppdownn\upp\nonee$-permutrees as in Figure~\ref{fig:permutree-inversion-vector-IXYI}. Then \begin{equation*}
		\begin{split}
			B(\hat{1})&=\{(1,2),(1,3),(1,4),(2,3),(2,4),(3,4)\},\\
			B(\hat{0})&=\emptyset,\\
			B(T_l)&=\{(2,4),(3,4)\},\\
			B(T_r)&=\{(1,2),(3,4)\}.
		\end{split}
	\end{equation*}
\end{example}

\begin{lemma}\label{lem:permutree_inversion_sets}
	Let~$1\leq i<j<k\leq n$. The set of inversion sets~$\{B(T)\,:\, T\in\cPT_n(\delta)\}$ is the set of all subsets~$E\subseteq \{(i,j)\in[n]^2\,:\,1\leq i<j\leq n\}$ such that \begin{enumerate}
		\itemsep0em
		\item $E$ is transitive,
		\item $E$ is cotransitive (i.e.\ the complement is transitive),
		\item if~$\delta_j\in\{\downn,\uppdownn\}$,~$(i,j)\notin E$, and~$(j,k)\in E$, then~$(i,k)\notin E$,
		\item if~$\delta_j\in\{\upp,\uppdownn\}$,~$(i,j)\in E$, and~$(j,k)\notin E$, then~$(i,k)\notin E$.
	\end{enumerate}
\end{lemma}

\begin{proof}
	Let~$T$ be a~$\delta$-permutree and~$E:=B(T)$. If~$(i,j),(j,k)\in E$, then~$j\rightarrow i,\,k\rightarrow j$, evidently~$k\rightarrow i$. That is,~$(i,k)\in E$ and~$E$ is transitive. The fact that~$E$ is cotransitive follows a similar argument. For property 3, the facts that~$(i,j)\notin E$ and~$(j,k)\in E$ respectively mean that~$v_i\in LD_j$ and~$v_k\in RD_j$. Thus,~$v_k$ is not a child of~$v_i$ and~$(i,k)\notin E$. Property 4 follows a similar argument.

	For the opposite direction we wish to construct a~$\delta$-permutree~$T(E)$ in accordance with the elements in~$E$. Let~$E_i=\{j\in [n]\,:\, (i,j)\in E\}$ be the components of~$E$. Notice that~$E_n=\emptyset$ and that if~$(i,j)\in E_i$, and~$(j,k)\in E_j$ then~$(i,k)\in E_i$ due to~$E$ being transitive. With this in mind, we can construct~$T(E)$ in the following way. Take an~$n\times n$ grid. As step~$0$, place vertex~$v_n$ anywhere in the last column. Now for step~$i$, place the vertex~$v_{n-i}$ in the~$n-i$-th column and at the height such that it is above (resp.\ below) all~$j$ such that~$(n-i,j)\in E$ (resp.~$(n-i,j)\notin E$). If such height was already used by another vertex, move the placed vertices up or down as required maintaining the relative others established in the previous steps. After step~$n-1$ we get a permutation table. Decorate each vertex~$v_i$ with the decoration~$\delta_i$. Following the insertion algorithm of~\cite{PP18} we obtain a~$\delta$-permutree~$T(E)$. Notice that the placement of the vertices in the grid ensures transitivity and cotransitvity and that the red walls from the decorations in the insertion algorithm accomplish properties 3 and 4. See Figure~\ref{fig:permutree_inversion_set_insertion_algorithm} for an example.
\end{proof}

\begin{figure}[h!]
	\centering
	\includegraphics[scale=0.7]{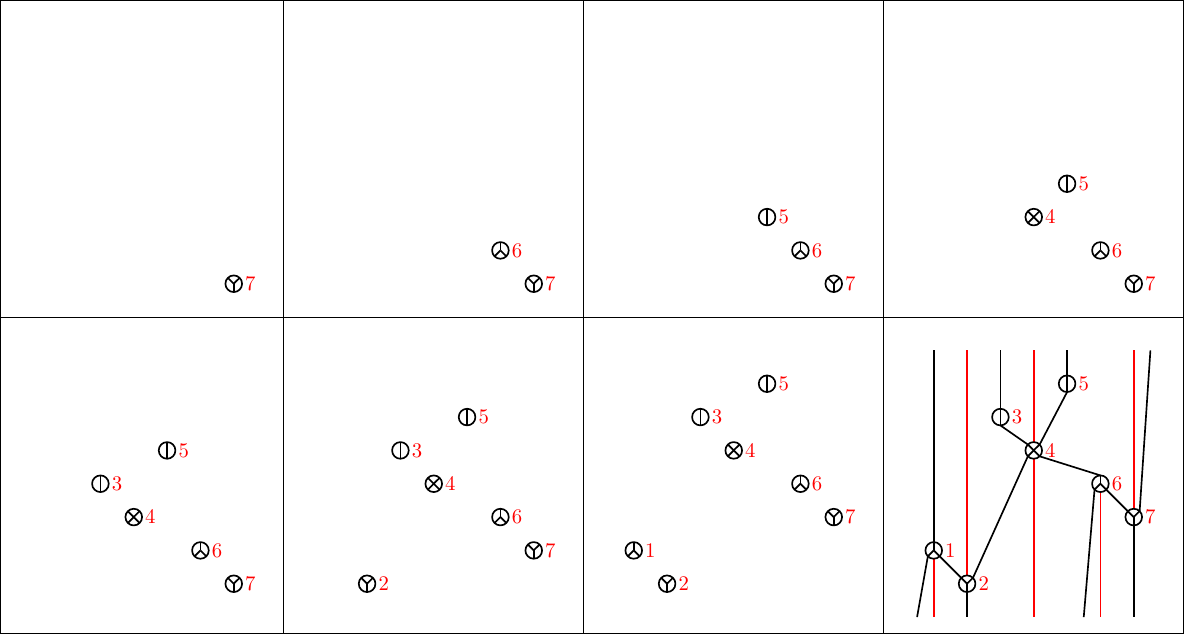}
	\caption{ The construction of the~$\downn\upp\nonee\uppdownn\nonee\downn\upp$-permutree corresponding to the inversion set~$\{(1,2),(3,4),(3,6),(3,7),(4,6),(4,7),(5,6),(5,7),(6,7)\}$.}\label{fig:permutree_inversion_set_insertion_algorithm}
\end{figure}

\begin{remark}\label{rem:inversion_set_inclusion}
	In the case of cover relations~$T\lessdot T'$ between~$\delta$-permutrees, that is, an~$ij$-edge rotation from~$T$ to~$T'$, Definition~\ref{def:permutree_rotations} tells us that such a rotation only affects the edge between~$v_i$ and~$v_j$ while the rest of the tree remains the same. In terms of inversions this means that the rotation only turns into inversions the pairs of the form~$(i,x)$ where~$x\in RD_j(T)$ and if~$\delta_i\in\{\upp,\uppdownn\}$, also all the pairs that depend on these in a transitive manner. That is,~$B(T')=(B(T)\cup \{(i,j)\})^{tc}$ no matter the decoration~$\delta$.
\end{remark}

\begin{remark}\label{rem:recovering_CPP_1}
	The characterization of inversion sets was already given in~\cite[Section 2.3.2]{CPP19} where they are called \emph{IPIP} (\emph{PIP} meaning permutree interval poset). With Lemma~\ref{lem:permutree_inversion_sets}, not only we have characterized inversion sets for permutrees but also described how to recover the permutree given its inversion set which is not done in~\cite{CPP19}.
\end{remark}

To use inversion sets as a tool we need first to show that we can describe the~$\delta$-permutree rotation order via their containment.

\begin{lemma}\label{lem:inversion_set_inclusion}
	Let~$T,T'$ be two~$\delta$-permutrees. Then~$T<T'$ if and only if~$B(T)\subset B(T')$.
\end{lemma}

\begin{proof}
	Suppose that~$T<T'$. Since the rotation order on permutrees is the transitive closure of the covering relations given by the rotations in Figure~\ref{fig:permutree_rotations}, it is enough to prove this in the case that~$T'$ covers~$T$. Remark~\ref{rem:inversion_set_inclusion} tells us that in such a case~$B(T')=(B(T)\cup \{(i,j)\})^{tc}$ and thus~$B(T)\subset B(T')$.

	Before moving to the other direction let~$\hat{0}$ be the minimal~$\delta$-permutree in the rotation lattice. The fact that~$T\lessdot T'$ implies that~$B(T)={(B(T)\cup\{(i,j)\})}^{tc}$ tells us that for any chain~$\hat{0}=T_0\lessdot T_1\lessdot\cdots\lessdot T_{l}\lessdot T_{l+1}=T$ in the interval~$[\hat{0},T]$, we have that a \defn{sequence} of inversions~$(i_x,j_x)$ such that~$B(T_x)={(B(T_{x-1})\cup\{(i_x,j_x)\})}^{tc}$ for all~$x\in[l]$. We say that such a sequence~$(i_1,j_1),\ldots,(i_l,j_l)$ \defn{generates}~$T$. It is easy to see that~$T<T'$ if and only if for every sequence~$(i_1,j_1),\ldots,(i_l,j_l)$ that generates~$T$ there exists a sequence of the form~$(i_1,j_1),\ldots,(i_l,j_l),(i_{l+1},j_{l+1}),\ldots,(i_{l'},j_{l'})$ that generates~$T'$. Take notice that the length of the chains in an interval of permutrees might not always be the same.

	Now suppose that~$B(T)\subset B(T')$ and let~$(i_1,j_1),\ldots,(i_l,j_l)$ be a sequence of~$T$. We claim that for all~$x\in\{0,\ldots,l\}$ the sequence~$(i_1,j_1),\ldots,(i_x,j_x)$ is the start of sequence of~$T'$. Let~$T_x$ correspond to the~$\delta$-permutree corresponding to the claim corresponding to~$x$. Notice that the claim for~$x=l$ amounts to proving~$T<T'$. We proceed by induction on the length of the chain which is given by~$x$. If~$x=0$ then~$T_0=\hat{0}$ and the claim is trivial. Now suppose that the claim holds for~$x-1$ and~$(i_1,j_1),\ldots,(i_{x-1},j_{x-1})$ is the start of a sequence of~$T'$, that is,~$T_{x-1}<T$. Suppose as well that the~$T_{x}\nless T'$. Since~$T_x< T'$, this means that with the~$i_xj_j$-edge rotation,~$T_{x}$ obtained an inversion that~$T'$ does not possess. This is a contradiction since~$B(T_{x})={(B(T_{x-1})\cup\{(i_x,j_x)\})}^{tc}\subset B(T)\subset B(T)$ as all sets~$B$ are transitive. Thus,~$T_x<T$ for all~$x\in\{0,\ldots,l\}$ and~$T<T'$.
\end{proof}

\begin{theorem}\label{thm:permutree_meet}
	Given two~$\delta$-permutrees~$T,T'$ on~$n$ vertices, there exists a~$\delta$-permutree~$T\wedge T'$ under the~$\delta$-permutree rotation order. Moreover, it satisfies \begin{equation}\label{eq:inversion_set_meet}
		B(T\wedge T')=B(T)\cap B(T')\cap \{(i,j)\,:\,\forall\,i<l<j,\, (i,l) \text{ or } (l,j)\in B(T)\cap B(T')\}.
	\end{equation}
\end{theorem}

Before proving Theorem~\ref{thm:permutree_meet} we give an example of how to compute the meet of two~$\delta$-permutrees and some remarks.

\begin{example}\label{ex:permutree_meet}
	Figure~\ref{fig:permutree_meet_example} presents the meet operation between~$\delta$-permutrees. Taking the last set of the right-hand side of Equation (\ref{eq:inversion_set_meet}) as~$I$, we have that the corresponding inversion sets in this case are \begin{equation*}
		\begin{split}
			B(T)&=           \{(2,3),(2,4),(2,5),(3,4)\}                         \\
			B(T')&=          \{(1,2),(1,4),(1,5),(2,4),(2,5),(3,4),(3,5),(4,5)\} \\
			I&=              \{(1,2),(1,3),(1,4),(2,3),(2,4),(3,4),(3,5),(4,5)\} \\
			B(T\wedge T')&=  \{(2,4),(3,4)\}
		\end{split}
	\end{equation*} and satisfy Theorem~\ref{thm:permutree_meet}. Translating this into the components of the inversion sets we have that \begin{multicols}{2}
		\noindent
		\begin{equation*}
			\begin{split}
				\emptyset\cap\{2,4,5\}\cap\{2,3\}&=\emptyset\\
				\{3,4,5\}\cap\{4,5\}\cap\{3,4\}&=\{4\}
			\end{split}
		\end{equation*}
		\begin{equation*}
			\begin{split}
				\{4\}\cap\{4,5\}\cap\{4,5\}&=\{4\}\\
				\emptyset\cap\{5\}\cap\{5\}&=\emptyset\\
			\end{split}
		\end{equation*}
	\end{multicols}
\end{example}

\begin{figure}[h!]
	\centering
	\includegraphics[scale=1]{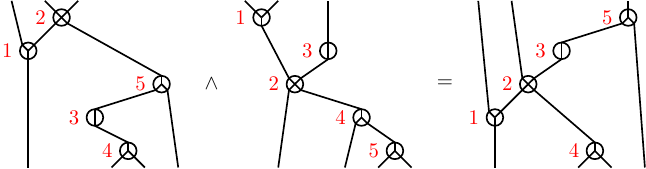}
	\caption{The meet of two~$\upp\uppdownn\nonee\downn\downn$-permutrees.}\label{fig:permutree_meet_example}
\end{figure}

\begin{remark}\label{rem:recovering_binary_meet_from_permutrees}
	Notice that inversion sets of~$\{\downn\}^n$-permutrees are bracket sets of binary trees as in Definition~\ref{def:bracket_vector}. We recover Proposition~\ref{prop:tamari_lattice_meet} whenever~$\delta\in\{\downn\}^n$ as~$B(T)\cap B(T')$ is contained in the last set of Equation (\ref{eq:inversion_set_meet}). To see this, notice that if~$(i,j)\in B(T)$ then~$v_l\in RD_i$ for all~$i<l<j$. If not, it would contradict that~$\delta_j=\downn$. Thus,~$(i,l)\in B(T)$. The same argument applies if~$(i,j)\in B(T')$ and thus in this case~$(i,j)\in B(T)\cap B(T')$ implies~$(i,l)\in B(T)\cap B(T')$ giving us the desired inclusion.
\end{remark}

\begin{remark}\label{rem:recovering_CPP_2}
	Our meet operation is similar to the \emph{tdd} operation defined in~\cite[Section 1.2.2]{CPP19} on integer posets. In this context, Theorem~\ref{thm:permutree_meet} can be seen as~\cite[Corollary 2.39]{CPP19} through the \emph{tdd} and our Lemma~\ref{lem:permutree_inversion_sets}. Nevertheless, we give here a direct proof without relying on the results of~\cite{CPP19}.
\end{remark}

We proceed to prove Theorem~\ref{thm:permutree_meet}.

\begin{proof}
	First let us see that~$B(T\wedge T')$ satisfies the conditions of Lemma~\ref{lem:permutree_inversion_sets} and thus defines~$T\wedge T'$ as a~$\delta$-permutree.

	Assume that~$(i,j),(j,k)\in B(T\wedge T')$. As~$(i,j),(j,k)\in B(T)\cap B(T')$, the transitivity of these sets tells us that~$(i,k)\in B(T)\cap B(T')$. We just need to see that~$(i,k)$ is also in the last set of Equation~\ref{eq:inversion_set_meet}. Let~$l$ such that~$i<l<k$. If~$i<l<j$, then as~$(i,j)\in B(T\wedge T')$ we have that either~$(i,l)\in B(T)\cap B(T')$ or~$(l,j)\in B(T)\cap B(T')$. In the former case we are done. In the latter, as~$(j,k)\in B(T)\cap B(T')$, using transitivity we get that~$(l,k)\in B(T)\cap B(T')$, and we are done. If instead~$l=j$, we immediately finish as by assumption~$(i,j),(j,k)\in B(T)\cap B(T')$. Finally, suppose that~$j<l<k$. In this case since~$(j,k)\in B(T\wedge T')$, either~$(j,l)\in B(T)\cap B(T')$ or~$(l,k)\in B(T)\cap B(T')$. In the latter case we finish. For the former, as~$(i,j)\in B(T)\cap B(T')$, using transitivity we get that~$(i,l)\in B(T)\cap B(T')$. Thus, we conclude that~$B(T\wedge T')$ is transitive.

	To see that~$B(T\wedge T')$ is cotransitive notice that its complement is the transitive closure of~$B(T)$ and~$B(T')$. That is,~$B(T\wedge T')^c=(B(T)^c\cup B(T')^c)^{tc}$. By definition of transitive closure it is immediate that~$B(T\wedge T')$ is cotransitive.

	Now suppose that~$\delta_j\in\{\downn,\uppdownn\}$,~$(i,j)\notin B(T\wedge T')$, and~$(j,k)\in B(T\wedge T')$. The last assumption tells us that~$(j,k)\in B(T)\cap B(T')$ and for all~$l$ such that~$j<l<k$, either~$(j,l)\in B(T)\cap B(T')$ or~$(l,k)\in B(T)\cap B(T')$. On the other hand, that~$(i,j)\notin B(T\wedge T')$ means that either~$(i,j)\notin B(T),\, (i,j)\notin B(T')$, or there exists~$i<l^*<j$ such that~$(i,l^*),(l^*,j)\notin B(T)\cap B(T')$. If either~$(i,j)\notin B(T)$ or~$(i,j)\notin B(T')$, then because of Property 3 of Lemma~\ref{lem:permutree_inversion_sets} and the fact that~$(j,k)\in B(T)\cap B(T')$ we have that~$(i,k)\notin B(T)$ and~$(i,k)\notin B(T')$ respectively. That is,~$(i,k)\notin B(T\wedge T')$ and we are done in this case.

	Consider then that~$(i,j)\in B(T)\cap B(T')$ and there exists~$i<l^*<j$ such that~$(i,l^*),(l^*,j)\notin B(T)\cap B(T')$. For contradiction’s sake suppose that~$(i,j)\in B(T\wedge T')$. By definition of~$B(T\wedge T')$ this means that either~$(i,l^*)\in B(T)\cap B(T')$ or~$(l^*,k)\in B(T)\cap B(T')$. The former case is a contradiction with the condition on which~$l^*$ exists, thus either~$(i,l^*)\in B(T)$ or~$(i,l^*)\in B(T')$ and the latter case happens. Without loss of generality suppose that~$(i,l^*)\in B(T)$. As~$(l^*,k)\in B(T)\cap B(T')\subset B(T)$, Property 3 of Lemma~\ref{lem:permutree_inversion_sets} tells us that~$(i,k)\notin B(T)$ and thus~$(i,k)\notin B(T\wedge T')$ as we wanted. The final Property of Lemma~\ref{lem:permutree_inversion_sets} follows a similar proof, and thus we omit it. We conclude that~$B(T\wedge T')$ indeed corresponds to a permutree~$T\wedge T'$.

	Let us now see that~$T\wedge T'$ is in fact the meet of~$T$ and~$T'$. Since~$B(T\wedge T')\subset B(T)$ and~$B(T\wedge T')\subset B(T')$ Lemma~\ref{lem:inversion_set_inclusion} tells us that~$T\wedge T'<T$ and~$T\wedge T'<T'$. Now suppose that there is a~$\delta$-permutree~$S$ such that~$S<T$ and~$S<T'$. We claim that~$S\leq T\wedge T$. Because of Lemma~\ref{lem:inversion_set_inclusion} we know that~$B(S)\subset B(T)\cap B(T')$. Let~$(i,j)\in B(S)\subset B(T)\cap B(T')$ and~$i<l<j$. Notice that if both elements~$(i,l),(l,j)\notin B(S)$, then~$(i,j)\notin B(S)$ as it is cotransitive, and we would have a contradiction. Without loss of generality suppose~$(i,l)\in B(S)$. As~$B(S)\subset B(T)\cap B(T')$, we have that~$(i,l)\in B(T)\cap B(T')$. Thus, for all~$i<l<j$ either~$(i,l)\in B(T)\cap B(T')$ or~$(l,j)\in B(T)\cap B(T')$. Meaning that,~$(i,j)\in B(T\wedge T')$ and we conclude that~$B(S)\subseteq B(T\wedge T')$. By Lemma~\ref{lem:inversion_set_inclusion} we get that~$S\leq T\wedge T'$ as we wished.
\end{proof}

\begin{corollary}\label{cor:permutrees_are_lattices}
	$\mathcal{PT}(\delta)$ is a lattice for any decoration~$\delta\in\{\nonee,\downn,\upp,\uppdownn\}^n$.
\end{corollary}

\begin{proof}
	The~$\delta$-tree rotation poset has a meet thanks to Theorem~\ref{thm:permutree_meet}. Since it is a bounded poset we get that it is a lattice~\cite[Prop.3.3.1]{S11},\cite[Lem.9-2.1]{R16}.
\end{proof}

\section{Cubic Vectors}\label{sec:cubic_vectors}

Having inversion vectors in hand, the reader might ask if it is the case that inversion vectors also give a cubic embedding of~$\delta$-permutree lattices. This is not the case as can be seen in Figure~\ref{fig:permutree_inversion_cubic_vectors}.

\begin{figure}[h!]
	\centering
	\includegraphics[scale=0.8]{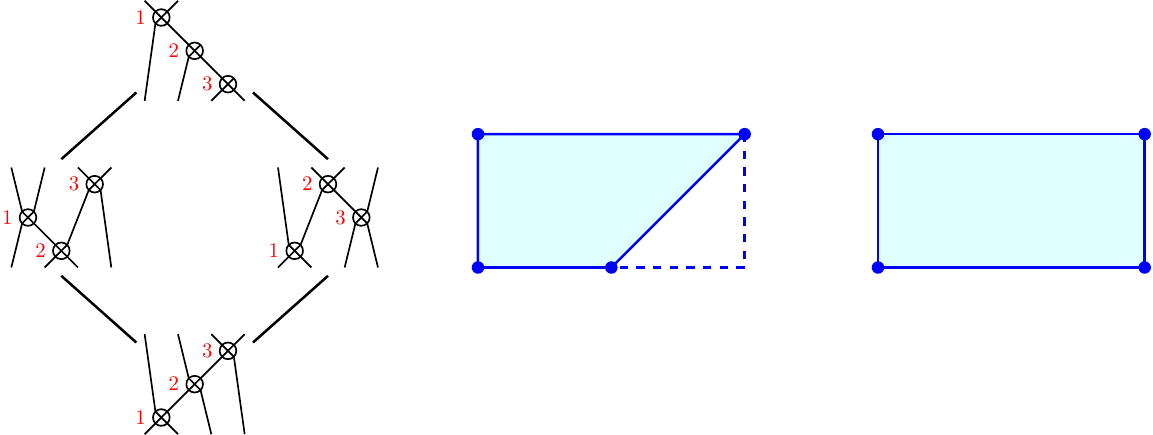}
	\caption{ The rotation lattice of~$\uppdownn\uppdownn\uppdownn$-permutrees (left) together with its geometric realizations using inversion vectors (middle) and cubic vectors (right). Taking~$T$ to be the middle-left~$\uppdownn\uppdownn\uppdownn$-permutree, we have~$\vec{b}(T)=(1,0)$ and~$\vec{c}(T)=(2,0)$.}\label{fig:permutree_inversion_cubic_vectors}
\end{figure}

One can still manage to get such an embedding, it suffices to slightly relax the definition of our sets.

\begin{definition}\label{def:cubic_vector_permutrees}
	Consider~$T\in\cPT_n(\delta)$ to be a~$\delta$-permutree. Its \defn{cubic set} is \begin{equation*}C(T):=\left\{(i,j)\,:\, \begin{array}{cc}
			i<j  \text{ and } v_j\in D_i & \text{ if } \delta_i\in\{\nonee,\upp\}      \\
			v_j\in RD_i                  & \text{ if } \delta_i\in\{\downn,\uppdownn\}\end{array} \right\}\end{equation*} and its \defn{cubic components} are~${C(T)}_i=\{j\in [n] \,:\, (i,j)\in C(T)\}$. A cubic set has an associated \defn{cubic vector}~$\vec{c}(T)=(c_1,\ldots,c_{n-1})$ such that~$c_i=|{C(T)}_i|$.
\end{definition}

\begin{remark}\label{rem:cubic_set_inclusion}
	Like in Remark~\ref{rem:inversion_set_inclusion}, we have that for a covering relation of~$\delta$-permutrees~$T\lessdot T'$ the respective cubic sets satisfy~$C(T')={(C(T)\cup\{(i,j)\})}^{tc}$. The key difference between the transitive closures of cubic vectors against inversion vectors is that the transitive closure turns into inversions the pairs of the form~$(i,x)$ where~$x\in RD_j(T)$ and nothing else. This is a consequence of the replacement of the condition~$j\to i$ in inversion sets to~$j\in RD_i$ (resp.~$i<j$ and~$j\in D_i$) in cubic sets.
\end{remark}

\begin{definition}\label{def:cubicrealizationpermutrees}
	Let~$T,T'\in\cPT_n(\delta)$. We say that there is an edge between~$\vec{c}(T)$ and~$\vec{c}(T')$ if and only if~$T\lessdot T'$. The convex hull of the cubic vectors together with this collection of edges is called the \defn{cubical realization}~$\cC_\delta$ of~$(\mathcal{PT}(\delta),\leq)$.
\end{definition}

\begin{example}\label{ex:cubic_realization_tamari_others}
	If~$\delta=\downn\downn\downn\downn$ (resp.~$\delta=\nonee\nonee\nonee\nonee$), the cubic vector reduces to the bracket vector of binary trees (resp.\ to the Lehmer code of permutations), and we recover the cubic realizations of the Tamari lattice in~\cite{C22} and~\cite{K93} and the weak order of~\cite{BF71} and~\cite{RR02}. See Figure~\ref{fig:permutreehedron_cubic} for these cubic realizations and other examples.
\end{example}

\begin{figure}[h!]
	\centering
	\includegraphics[scale=1]{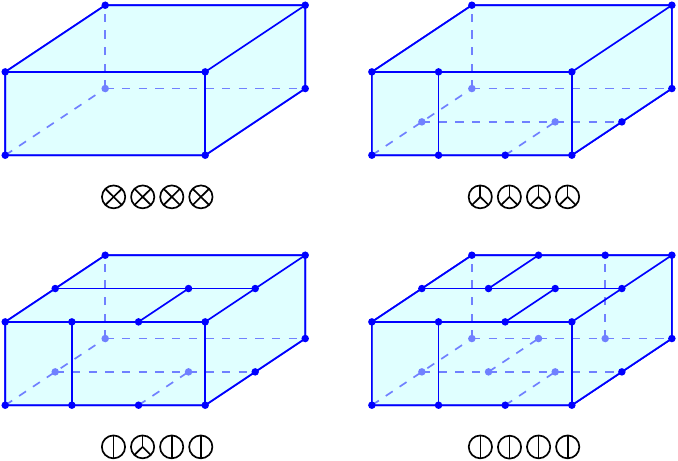}
	\caption{The Cubical realization~$\cC_\delta$ of several permutreehedra.}\label{fig:permutreehedron_cubic}
\end{figure}

We now enunciate several properties of cubical realizations which culminate in showing that~$\cC_\delta$ is an embedding of~$\PPT_n(\delta)$ into the cube~$Q_n=[0,n-1]\times\cdots\times[0,1]$.

\begin{theorem}\label{thm:cubic_property_edge_direction}
	If~$T,T'\in\cPT_n(\delta)$ are~$\delta$-permutrees such that~$T< T'$ in the~$\delta$-permutree rotation lattice, then~$\vec{c}(T)<_{lex} \vec{c}(T')$ and the edges of~$C_\delta$ have directions~$\mathbf{e_i}$.
\end{theorem}

\begin{proof}
	We prove it for covering relations~$T\lessdot T'$ as the~$\delta$-rotation order is the transitive closure of the relations in Figure~\ref{fig:permutree_rotations}. Following Remark~\ref{rem:cubic_set_inclusion} we have that the inclusion~$C(T')=(C(T)\cup\{(i,j)\})^{tc}$ and the relations between components~$C(T)_i\subset C(T')_i$ and~$C(T)_j= C(T')_j$ if~$j\neq i$. This tells us that~$\vec{c}(T')-\vec{c}(T)=(0,\ldots,0,c(T')_i-c(T)_i,0,\ldots,0)$. Therefore,~$\vec{c}(T)\leq_{lex}\vec{c}(T')$ and the edge~$[\vec{c}(T),\vec{c}(T')]$ has direction~$\mathbf{e_i}$.
\end{proof}

\begin{remark}
	The converse of Theorem~\ref{thm:cubic_property_edge_direction} is true for binary trees but not for permutrees in general. Note that in $\mathcal{C}_{\nonee\nonee\nonee\nonee}$ of Figure~\ref{fig:permutreehedron_cubic} $(1,0,1)\leq_{lex}(1,1,1)$, but there is no edge connecting them.
\end{remark}

\begin{theorem}\label{thm:cubic_property_convec_cube}
	$\cC_\delta$ is normal equivalent to~$\PCube_{n-1}$ (i.e. has the same normal fan).
\end{theorem}

\begin{proof}
	Let~$Q_{n-1}:=[0,n-1]\times\cdots\times[0,1]$. First note that~$0\leq c(T)_i\leq n-i$ for all~$i\in [n-1]$ meaning that~$\mathcal{C}_{\delta}\subset Q_{n-1}$. To see the reverse inclusion it is enough to prove that all vectors~$\vec{r}=(r_1,\ldots,r_{n-1})$ where~$r_i\in\{0,n-i\}$, have a preimage through the function~$f:\cPT_n(\delta)\to \cC_\delta$ such that~$f(T)=\vec{c}(T)$. We call such preimages \defn{extremal}~$\delta$-permutrees.

	Take any such~$\vec{r}$. We now present how to construct a~$\delta$-permutree in the preimage~$f^{-1}(\vec{r})$. Consider an~$n\times n$ grid. At step~$1$ place~$v_1$ at~$(1,1)$ (resp.~$(1,n)$) if~$r_1=0$ (resp.~$r_1=n-1$). At step~$i$ place~$v_i$ at~$(i,d)$ (resp.~$(i,n-u)$) where~$d:=|\{j\in[n]\,\: j<i \text{ and } r_j=0\}|$ (resp.~$u:=|\{j\in[n]\,\: j<i \text{ and } r_j=n-j\}|$). After step~$n-1$ place~$v_n$ in the only coordinate of column~$n$ that shares no vertex horizontally. Thus, we get a permutation table. Decorate each vertex~$v_i$ with the decoration~$\delta_i$. Following the insertion algorithm of~\cite{PP18} we obtain a~$\delta$-permutree~$T(\vec{r})$.

	Notice that in~$T(\vec{r})$, for each vertex~$v_i$ we have either~$|RD_i|=0$ (resp.~$|\{j\in[n]\,:\,i<j \text{ and } v_j\in D_i\}|=0$) or~$|RD_i|=n-i$ (resp.~$|\{j\in[n]\,:\,i<j \text{ and } v_j\in D_i\}|=n-i$) That is, the values corresponding to~$\vec{r}$. Therefore,~$\cC_\delta=Q_{n-1}$ and normal equivalent to~$\PCube_{n-1}$.
\end{proof}

\begin{remark}\label{rem:cubic_embedding_interior_points}
	Notice that since the interior of~$Q_{n-1}$ has no integer points, we have that all cubic coordinates are on the surface of~$\mathcal{C}_{\delta}$.
\end{remark}

Figure~\ref{fig:permutree_cubic_insertion_algorithm} shows an example of the construction of extremal permutrees described in the proof of Theorem~\ref{thm:cubic_property_convec_cube}. In Figure~\ref{fig:permutree-inversion-vector-IXYI} the extremal~$\nonee\uppdownn\upp\downn$-permutrees are colored in black while the~$2$ that are extremal are colored in brown. We now show that these preimages are unique as a part of the following bigger result.

\begin{figure}[h!]
	\centering
	\includegraphics[scale=0.8]{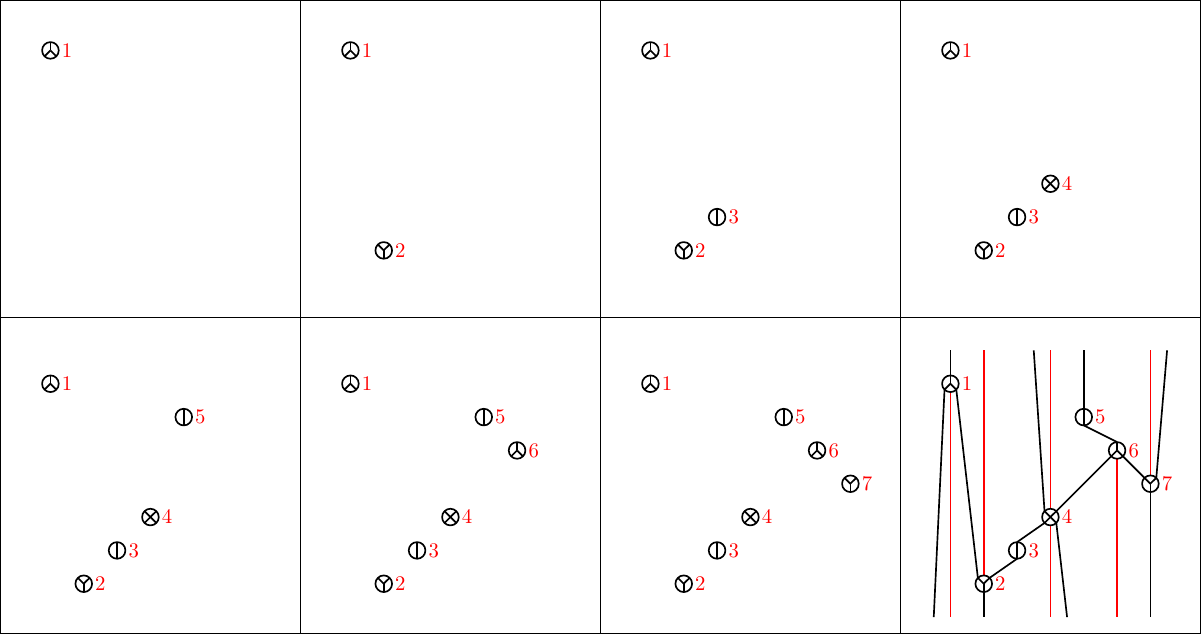}
	\caption{ The construction of the extremal~$\downn\upp\nonee\uppdownn\nonee\downn\upp$-permutree corresponding to the corner~$(6,0,0,0,2,1)\in Q_{6}$.}\label{fig:permutree_cubic_insertion_algorithm}
\end{figure}

\begin{theorem}\label{thm:cubic_property_injective}
	The map~$f:\cPT_n(\delta)\to\cC_\delta$ sending a~$\delta$-permutree to its cubic vector is injective.
\end{theorem}

\begin{proof}
	Consider~$T,T'\in \mathcal{PT}_{\delta}$ two different~$\delta$-permutrees. Due to them being different, there is a maximal vertex~$i$ such that~$RD(T)_i\neq RD(T')_i$ (resp.~$\{j\in[n]\,:\, i<j \text{ and } v_j\in D(T)_i\}\neq\{j\in[n]\,:\, i<j \text{ and } v_j\in D(T')_i\}$). If~$c(T)_i=|RD(T)_i|\neq|RD(T')_i|=c(T')_i$ (or the equivalent in the~$D_i$ case) we are done. Otherwise, there exists a maximal vertex~$r\in RD(T)_i\cap RD(T')_i$ such that~$RD(T)_r\neq RD(T')_r$ which contradicts the existence of~$i$. Therefore,~$\vec{c}(T)\neq\vec{c}(T')$.
\end{proof}

\begin{theorem}\label{thm:cubic_property_embedding}
	$\mathcal{C}_\delta$ is an embedding of the~$\delta$-permutreehedron. In particular, maximal cells of~$\mathcal{C}_{\delta}$ are in bijection with facets of the~$\delta$-permutreehedron.
\end{theorem}

\begin{proof}
	Recall from Subsection~\ref{ssec:permutrees} that the facets of the~$\delta$-permutreehedron are in bijection with the proper subsets~$I\subsetneq[n]$ such that there is a~$\delta$-permutree that admits~$(I\,\|\,{[n]}\setminus {I})$ as an edge cut. Let~$J:=[n]\setminus I$. As~$\delta$-permutrees are connected, edge cuts partition a~$\delta$-permutree~$T$ into a~$\delta_I$-permutree~$T_I$ and a~$\delta_J$-permutree~$T_J$ which as subtrees are connected only via an edge~$(i,j)$ such that~$i\to j$ and~$v_i\in T_i$ and~$v_j\in T_j$.

	Take an edge-cut~$(I\,\|\, J)$. We proceed to construct a cell~$K$ of~$\mathcal{C}_\delta$ (coming from a polygonal interval of the rotation lattice) containing all cubic vectors~$\vec{c}(T)$ of~$\delta$-permutrees~$T$ that admit said edge-cut. Consider the minimal elements~$\hat{0}_{\delta_I}$ and~$\hat{0}_{\delta_J}$ (resp. maximal elements~$\hat{1}_{\delta_I}$ and~$\hat{1}_{\delta_J}$). Connecting them via the insertion algorithm gives us the~$(I\,\|\, J)$-admitting~$\delta$-permutree~$\underline{T}$ given by~$\underline{T}_I:=\hat{0}_{\delta_I}$ and~$\underline{T}_J:=\hat{0}_{\delta_J}$ (resp.~$\overline{T}$ given by~$\overline{T}_I:=\hat{1}_{\delta_I}$ and~$\overline{T}_J:=\hat{1}_{\delta_J}$). Notice that~$\underline{T}$ (resp.~$\overline{T}$) is the minimal (resp.\ maximal)~$\delta$-permutree that admits~$(I\,\|\, J)$ as an edge cut. This in turn shows that~$\vec{c}(\underline{T})$ (resp.~$\vec{c}(\overline{T})$) is the lexicographical minimal (resp.\ maximal) cubic vector that relate with this edge-cut. Thus, we define our cell as~$K:=\{\vec{c}(T)\in\mathcal{C}_\delta\,:\,\underline{T}\leq T\leq\overline{T}\}$.

	Let us see that~$K$ is maximal by showing it is contained in a hyperplane. Suppose that~$n\in J$. In such case, for any~$\delta$-permutree~$T$ such that~$\underline{T}\leq T\leq\overline{T}$ we have that~$RD(T)_{\max(I)}=\emptyset$ (resp.~$\{j\in[n]\,:\, \max(I)<j \text{ and } v_j\in D(T)_{\max(I)}\}=\emptyset$) and we conclude that get that~$K$ is in the hyperplane~$x_{\max(I)}=0$. If instead~$n\in I$, then we obtain that~$K$ is in the hyperplane~$x_{\max(J)}=n-\max(J)$ following a similar argument. Finally, note that all other~$n-2$ entries of the cubic vectors change between~$\underline{T}$ and~$\overline{T}$ through rotations between the vertices~$I$ or~$J$. This together with Theorem~\ref{thm:cubic_property_edge_direction} gives us that~$K$ is a maximal cell of~$\mathcal{C}_\delta$.

	The conjunction of Theorems~\ref{thm:cubic_property_edge_direction},~\ref{thm:cubic_property_convec_cube}, and~\ref{thm:cubic_property_injective} together with our bijection between facets and cells gives us that~$\mathcal{C}_\delta$ is an embedding of the~$\delta$-permutreehedron.
\end{proof}

\section*{Acknowledgments}

The author is thankful to Viviane Pons and Vincent Pilaud for proposing this problem and their guidance during the writing and development of this paper. The author also thanks Camille Combe, Jean-Philippe Labbé, and Clément Cheneviere for several interesting discussions, the LIGM team of Université Paris-Est Marne-la-Vallée for several readings of a previous extended abstract version, and Jose Bastidas and an anonymous referee for useful comments. The author was partially supported for this project by the GALaC team at the LISN (Université Paris-Saclay) and the department of mathematics of the Universidad de los Andes.

\bibliographystyle{alpha}
\bibliography{bibliography.bib}

\Addresses{}

\end{document}